\documentclass{article}
\usepackage{mathrsfs}
\usepackage{amsfonts}
\usepackage{epsfig}
\usepackage{enumerate}
\usepackage{graphics}
\usepackage{cite}
\usepackage{authblk}
\usepackage[left=3cm,right=3cm,top=2.5cm,bottom=2.5cm]{geometry}
\usepackage{amssymb}
\usepackage{amsmath,amsthm}
\usepackage{caption}
\newtheorem{theorem}{Theorem}[section]

\newtheorem{corollary}[theorem]{Corollary}
\newtheorem{proposition}[theorem]{Proposition}
\newtheorem{definition}[theorem]{Definition}

\newtheorem{remark}[theorem]{Remark}
\numberwithin{equation}{section}

\begin{document}
\title{An Invariant of Algebraic Curves from the Pascal Theorem\thanks{The project is supported by NNSFC(Nos. 10771028, 60533060),
Program of New Century Excellent Fellowship of NECC, and is
partially funded by a DoD fund (DAAD19-03-1-0375).}}
\author{Zhongxuan Luo\thanks{Corresponding author: zxluo@dlut.edu.cn}}
\affil{School of Software, School of Mathematical Sciences, Dalian University of
Technology, Dalian, 116024, China} 



\date{Sept 5, 2007}

\maketitle

\begin{abstract}
In 1640's, Blaise Pascal discovered a remarkable property of a
hexagon inscribed in a conic - Pascal Theorem, which gave birth of
the projective geometry. In this paper, a new geometric invariant of
algebraic curves is discovered by a different comprehension to
Pascal¡¯s mystic hexagram or to the Pascal theorem. Using this
invariant, the Pascal theorem can be generalized to the case of
cubic (even to algebraic curves of higher degree), that is, {\em For
any given 9 intersections between a cubic $\Gamma_3$ and any three
lines $a,b,c$ with no common zero, none of them is a component of
$\Gamma_3$, then the six points consisting of the three points
determined by the Pascal mapping applied to any six points (no three
points of which are collinear) among those 9 intersections as well
as the remaining three points of those 9 intersections must lie on a
conic.} This generalization differs quite a bit and is much simpler
than Chasles's theorem and Cayley-Bacharach theorems.
\\[6pt]
\textbf{Keywords:}
 Algebraic curve; Pascal theorem; Characteristic ratio;
Characteristic mapping; Characteristic number; spline.
\end{abstract}

\markboth{An Invariant from the Pascal Theorem} {Zhongxuan Luo}

\vspace{-5mm}
\section{Introduction}

Algebraic curve is a classical and an important subject in algebraic
geometry. An algebraic plane curve is the solution set of a
polynomial equation $P(x,y)=0$, where $x$ and $y$ are real or
complex variables, and the degree of the curve is the degree of the
polynomial $P(x,y)$. Let $\mathbb{P}^2$ be the projective plane and
$\mathbb{P}_n$ be the space of all homogeneous polynomials in
homogeneous coordinates $(x,y,z)$ of total degree $\leq n$. An
algebraic curve $\Gamma_n$ in the projective plane is defined by the
solution set of a homogeneous polynomial equation $P(x,y,z)=0$ of
degree $n$.

In 1640, Blaise Pascal discovered a remarkable property of a hexagon
inscribed in a circle, shortly thereafter Pascal realized that a
similar property holds for a hexagon inscribed in an ellipse even a
conic. As the birth of the projective geometry, Pascal theorem
assert: {\em If six points on a conic section is given and a hexagon
is made out of them in an arbitrary order, then the points of
intersection of opposite sides of this hexagon will all lie on a
single line.} The generalizations of Pascal's theorem have a
glorious history. It has been a subject of active and exciting
research. As generalizations of the Pascal theorem, Chasles's
theorem and Cayley-Bacharach theorems in various versions received a
great attention both in algebraic geometry and in multivariate
interpolation.  A detailed introduction to Cayley-Bacharach theorems
as well as conjectures can be found in \cite{EGH96,G-J,S-R}.

%

The Pascal theorem can be comprehended in the following several
aspect: first, it is easy to verify that Pascal's theorem can be
proved by Chasles's theorem\cite{EGH96} and therefore, probably,
Chasles's theorem has been regarded as a generalization of Pascal's
theorem in the literature. However, the Chasles's theorem and
Cayley-Bacharach theorems have not formally inherited the appearance
of the Pascal theorem, that is the three points joined a line are
obtained from intersections of three pair of lines in which each
line was determined by two points lying on a conic; Secondly, the
Pascal theorem can be used to (geometrically) judge whether or not
any six points simultaneously lie on a conic. Another interesting
observation to the Pascal theorem is that it plays a key role in
revealing the instability of a linear space $S_2^1(\Delta_{MS})$
(the set of all piecewise polynomials of degree 2 with global
smoothness 1 over Morgan-Scott triangulation, see figure 5.1 and
refer to Appendix 5.1). That is, the Pascal theorem gives an
equivalent relationship between the algebraic and geometric
conditions to the instability of $S_2^1(\Delta_{MS})$(refer to
Appendix 5.1).

Actually, readers will see that the Pascal theorem contains a
geometric invariant of algebraic curves, which is exactly reason
that we rake up the Pascal theorem in this paper. In order to get
this new invariant of algebraic curves, one must observe the Pascal
theorem from a different viewpoint in which ``arbitrary six points
are given by intersections of a conic and any three lines without no
common zero" instead of ``six points on a conic section is given" (a
historical viewpoint) in the Pascal theorem. This slight different
comprehension to the Pascal theorem makes us easily generalize the
Pascal theorem to algebraic curves of higher degrees and discover an
invariant of algebraic curves. Similar to the source of this paper
in which all involved points are the set of intersections between
lines and a curve, \cite{KG2006} has given some interesting results
to the special case of the following classical problem: Let $X$ be
the intersection set of two plane algebraic curves $\mathcal {D}$
and $\mathcal {E}$ that do not share a common component. If $d$ and
$e$ denote the degrees of $\mathcal {D}$ and $\mathcal {E}$,
respectively, then $X$ consists of at most $d\cdot e$ points (a week
form of Bezout's theorem\cite{Walker50}). When the cardinality of
$X$ is exactly $d\cdot e$, $X$ is called a complete intersection.
How does one describe polynomials of degree at most $k$ that vanish
on a complete intersection $X$ or on its subsets? The case in which
both plane curves $\mathcal {D}$ and $\mathcal {E}$ are simply
unions of lines and the union $\mathcal {D} \cup \mathcal {E}$ is
the $(d\times e)$-cage in question.

Our main results in this paper are enlightened by studying the
instability of spline space and are proved by spline method and the
``principle of duality " in the projective plane. This paper is
organized as follows: In section 2, some basic preliminaries of the
projective geometry are given. In section 3, some new concepts such
as characteristic ratio, characteristic mapping and characteristic
number of algebraic curve are introduced by discussing the
properties of a line and a conic. Section 4 gives our main results
for the invariant of cubic and presents a generalization of the
Pascal type theorem to cubic. Moreover, some corresponding
conclusions to the case of algebraic curves of higher degrees
$(n>3)$ are also stated in this section without proofs. The basic
theory of bivariate spline , a series of results on the singularity
of spline space and the proof of the main result of this paper are
given in Appendix in the end of the paper.

\section{Preliminaries of Projective Geometry}

It is well known that the ``homogeneous coordinates" and the ``
principle of duality"\footnote{Poncelet claimed this principle as
his own discovery, but its nature was more clearly understood by
another Franchman, J. D. Gergonne(1771-1859)\cite{Coxeter}. } are
the essential tools in the projective geometry. A point is the set
of all triads equivalent to given triad $(x)=(x_1,x_2,x_3)$, and a
line is the set of all triads equivalent to given triad
$[X]=[X_1,X_2,X_3]$. By a suitable multiplication (if necessary),
any point in the projective plane can be expressed in the form
$(x_1,x_2,1)$, which can be shortened to $(x_1,x_2)$, and the two
numbers $x_1$ and $x_2$ are called the affine coordinates. In other
words, if $x_3\neq 0$, the point $(x_1,x_2,x_3)$ in the projective
plane can be regarded as the point $(x_1/x_3,x_2/x_3)$ in the affine
plane. The ``principle of duality" in the projective plane can be
seen clearly from the following result: "three points $(u),(v)$ and
$(w)$ in $\mathbb{P}^2$ are collinear" is equivalent to "three lines
$[u],[v]$ and $[w]$ in $\mathbb{P}^2$ are concurrent". In fact, the
necessary and sufficient condition for the both statements is: there
are numbers $\lambda, \mu, \nu$, not all zero, such that $\lambda
u_i+\mu v_i +\nu w_i=0 (i=1,2,3),$ namely,
\begin{eqnarray*}
\left|
  \begin{array}{ccc}
    u_1 & u_2 & u_3 \\
    v_1 & v_2 & v_3 \\
    w_1 & w_2 & w_3 \\
  \end{array}
\right|=0.
\end{eqnarray*}
If $(u),(v)$ are distinct points, $\nu \neq 0.$ Hence the general
point collinear with $(u)$ and $(v)$ can be formed a linear
combination of $(u)$ and $(v)$. In other word, a point
$(u)=(u_1,u_2,u_3)\in \mathbb{P}^2$ corresponds uniquely to a line
$[u]=[u_1,u_2,u_3]:u_1x+u_2y+u_3z=0$, while a line
$[u]=[u_1,u_2,u_3]:u_1x+u_2y+u_3z=0$ corresponds uniquely a point
$(u)=(u_1,u_2,u_3)$. We say that a point $(u)$ and the corresponding
line $[u]$ are dual to each other - which is the two-dimensional
``principle of duality". Under this duality, it follows the
following definition.
\begin{definition}[Duality of planar figure] Let $\Delta$ be a
planar figure consisting of lines and points in the projective
plane. A planar figure obtained by the corresponding dual lines and
points of the points and lines in $\Delta$ respectively is called
the {\em Dual} figure of $\Delta$, denotes by $\Delta^*$.
\end{definition}
For instance, the dual figure of Fig. \ref{draftfigure} is shown in
Fig. \ref{linespoints}, where $[\cdot]$ represents the corresponding
dual line of the point $(\cdot)$ in Fig. \ref{draftfigure}.

\begin{figure}[h]
\begin{minipage}[b]{0.45\textwidth}
\centering
\includegraphics[width=\textwidth]{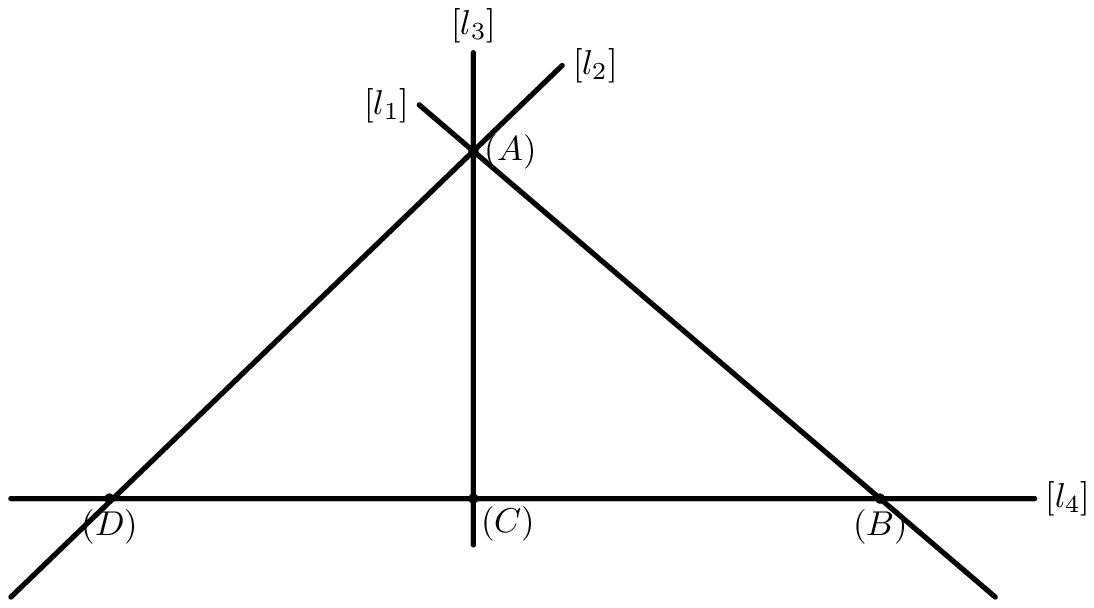}\\
\label{draftfigure}
\caption{}
\end{minipage}
\begin{minipage}[b]{0.4\textwidth}
\centering
\includegraphics[width=0.6\textwidth]{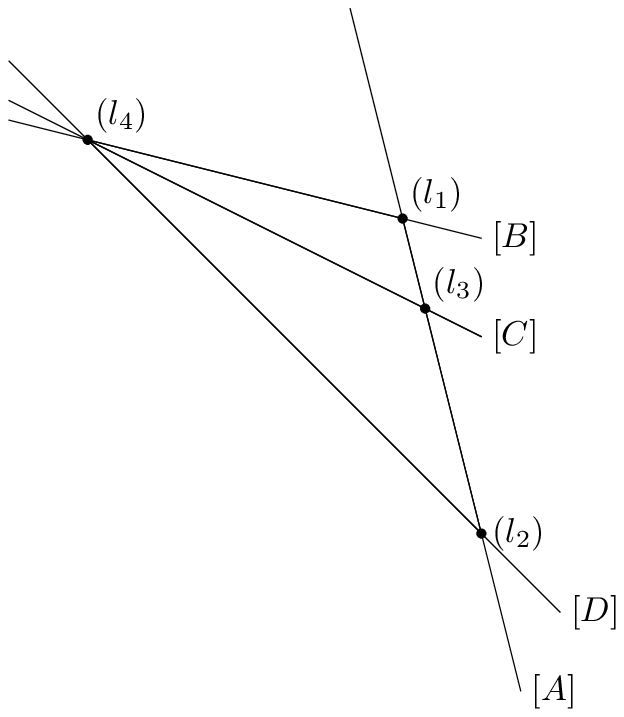}\\
\label{linespoints}
\caption{}
\end{minipage}
\end{figure}

\section{New Definitions}

In what follows, we shall use $u$ to represent a point $(u)$ or a
line $[u]$ when no ambiguities exist, $u=<a,b>$ for the intersection
point of lines $a$ and $b$, and $a=(u,v)$ for the line which joins
the points $u$ and $v$.
\begin{figure}[h]
\begin{minipage}{0.4\textwidth}
\centering
\includegraphics[width=\textwidth]{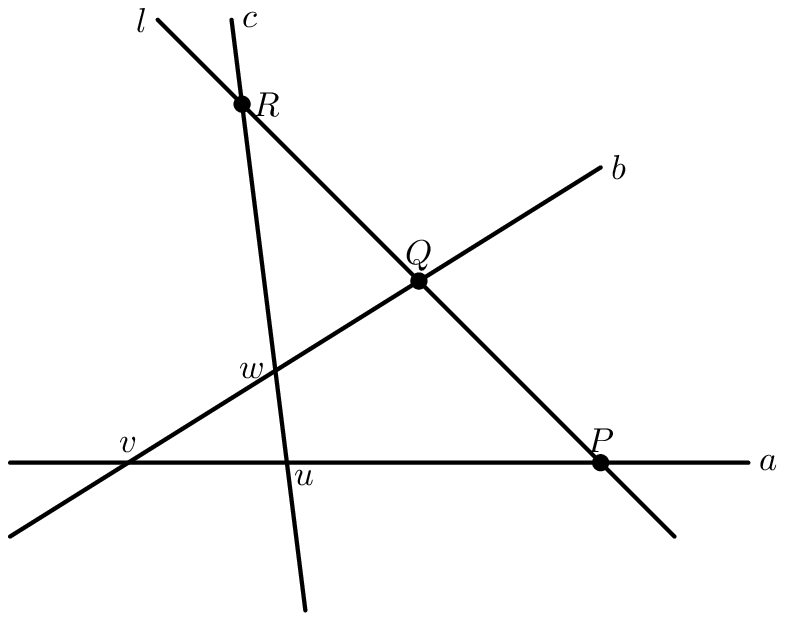}\\
\label{cutbyline} \caption{}
\end{minipage}
\begin{minipage}{0.58\textwidth}
First, we review the following properties of a line and a conic.
Suppose a line $l$ be cut by any three lines $a, b$ and $c$ with no
common zero (see Fig. \ref{cutbyline}). Let $u=<c,a>,v=<a,b>$ and
$w=<b,c>$, $P=<l,a>, Q=<l, b>$ and $R=<l,c>$. Obviously, there exist
numbers $a_i,b_i\ \ (i=1,2,3)$ such that $P=a_1u+b_1v, Q=a_2v+b_2w,
R=a_3w+b_3u,$ provided in turn $u,v,w$, then we have

\begin{proposition}
$$\frac{b_1}{a_1}\cdot \frac{b_2}{a_2}\cdot \frac{b_3}{a_3}=-1.$$
\end{proposition}
\notag
\end{minipage}
\end{figure}

\begin{proof} Without loss of generality, we assume
that $u=(1,0,0),v=(0,1,0)$ and $w=(0,0,1)$. Since $P, Q$ and $R$ are
collinear, hence
\begin{displaymath}
\left|
\begin{array}{ccc}
a_1 & b_1 & 0 \\
0   & a_2 & b_2\\
b_3 & 0   & a_3
\end{array}
\right|=0.
\end{displaymath}
It follows that $\frac{b_1}{a_1}\cdot \frac{b_2}{a_2}\cdot
\frac{b_3}{a_3}=-1.$
\end{proof}
With the same notations, it follows that the necessary and
sufficient condition for $P,Q$ and $R$ to be collinear is
$\frac{b_1}{a_1}\cdot \frac{b_2}{a_2}\cdot \frac{b_3}{a_3}=-1$.

Now let us replace the line $l$ in proposition 3.1 by a conic
$\Gamma$. There are two intersections between $\Gamma$ and each
$a,b,c$. Let $\{p_1, p_2\}=<\Gamma, a>$, $\{p_3, p_4\}=<\Gamma, b>$
and $\{p_5, p_6\}=<\Gamma, c>$. Consequently, there are real numbers
$\{a_i,b_i\}_{i=1}^6$ such that
\begin{eqnarray}
\left\{%
\begin{array}{ll}
    p_1=a_1u+b_1v \\
    p_2=a_2u+b_2v
\end{array}%
\right. ,
  \left\{%
\begin{array}{ll}
    p_3=a_3v+b_3w \\
    p_4=a_4v+b_4w
\end{array}%
\right.  \mbox{and}
 \left\{%
\begin{array}{ll}
    p_5=a_5w+b_5u \\
    p_6=a_6w+b_6u
\end{array}%
\right..
\end{eqnarray}

We have
\begin{theorem}
Let a conic be cut by any three lines with no common zero. Under the
notations above, we have
\begin{eqnarray}
\frac{b_1b_2}{a_1a_2}\cdot \frac{b_3b_4}{a_3a_4}\cdot
\frac{b_5b_6}{a_5a_6}=1.
\end{eqnarray}
\end{theorem}
\begin{proof} Let $u=<c,a>,v=<a,b>$,
$w=<b,c>$. Notice that the duality of the figure composed of the
points $\{p_i\}_{i=1}^6$, $u,v,w$ and the lines $a,b,c$ turns out a
planar figure with a structure of Morgan-Scott triangulation with
inner edges consists of the dual lines of the points
$\{p_i\}_{i=1}^6$, $u,v,w$ (see Fig. 6). Note that the six points
$\{p_i\}_{i=1}^6$ lie on a conic, it is shown from Theorem 5.4 (see
appendix 5.1) that the spline space $S_2^1(\Delta_{MS})$ (the set of
all piecewise polynomial of degree 2 with smoothness 1 over
Morgan-Scott triangulation $\Delta_{MS}$) is singular, that is $\dim
S_2^1(\Delta_{MS})=7$. Which implies from Theorem 5.5 (see appendix
5.1) that Theorem 3.3 thus follows.
\end{proof}

On the other hand, Theorem 3.2 can be used to tell whether or not
any six points simultaneously lie on a conic. In fact, let $p_i \in
\mathbb{P}^2 (i=1,2,\cdots,6)$ be any six distinct points without
any three points are collinear, $a=(p_1,p_2), b=(p_3,p_4)$,
$c=(p_5,p_6)$, and $u=<c,a>,v=<a,b>$, $w=<b,c>$. Using the same
notations as in (3.2), it follows from the proof of Theorem 3.2 that
\begin{proposition} For any given six points $p_1,p_2,\cdots,p_6$ without
no three points are collinear, (3.2) is a necessary and sufficient
condition for those six points to be lying on a conic.
\end{proposition}

Actually, Theorem 3.2 is equivalent to the Pascal theorem.
\\
\textbf{[Proof of Pascal theorem.]} Let $p_i \in \mathbb{P}^2
(i=1,2,\cdots,6)$ be any six distinct points without any three
points are collinear. Denoted by $a=(p_1,p_2), b=(p_3,p_4),
c=(p_5,p_6)$ and $u=<c,a>,v=<a,b>$, $w=<b,c>$. Without loss of
generality, we assume $u=(1,0,0),v=(0,1,0)$ and $w=(0,0,1)$. Since
the 6 points $p_i \in \mathbb{P}^2 (i=1,2,\cdots,6)$ lie on a conic,
(3.2) holds. It is clear that
\begin{center}
$
\begin{array}{ll}
q_1=<(p_1,p_2),(p_{4},p_{5})>=(b_{4}b_{5},-a_{4}a_{5},0)=b_{4}b_{5}u-a_{4}a_{5}w,\\
q_2=<(p_2,p_{3}),(p_{5}p_{6})>=(a_2a_{3},0,-b_2b_{3})=-b_2b_{3}v+a_2a_{3}u,\\
q_3=<(p_{3},p_{4}),(p_1,p_{6})>=(0,-b_1b_{6},a_1a_{6})=-b_1b_{6}w+a_1a_{6}v,\\
\end{array}
$
\end{center}
and (3.2) is equivalent to
\begin{eqnarray}
(-\frac{
b_1b_6}{a_1a_6})\cdot(-\frac{b_2b_3}{a_2a_3})\cdot(-\frac{b_4b_5}{a_4a_5})
=-1.
\end{eqnarray}
By Proposition 3.1, three points $\{q_1,q_2,q_3\}$ must be
collinear. This is the conclusion of the Pascal theorem.

Notice that Proposition 3.1 and Theorem 3.2, -1 and 1 are invariants
of line and conic respectively. We therefore introduce the following
definitions.

\begin{definition}[Characteristic ratio] Let $u,v \in
\mathbb{P}^2$ be two distinct points (or lines),
$p_1,p_2,\cdots,p_k$ be points (or lines) on the line $(u,v)$ (or
passing through $<u,v>$), then there are numbers $a_i, b_i$ such
that $p_i=a_iu+b_iv, i=1,2,\cdots,k$. The ratio
$$[u,v;p_1,\cdots,p_k]:=\frac{b_1b_2\cdots,b_k}{a_1a_2\cdots,a_k}$$
is called the \textbf{\em Characteristic ratio} of
$p_1,p_2,\cdots,p_k$ with respect to the basic points (or lines)
$u,v$. If there are multiple points in the intersection points, the
corresponding characteristic ratio is defined by their limit form.
\end{definition}
\begin{remark}
For four collinear points $u,v,p_1,p_2$, while the Characteristic
ratio of $p_1,p_2$ with respective to $u,v$ is
$\frac{b_1b_2}{a_2b_2}$, the cross ratio in the projective geometry
is defined as $\frac{a_1b_2}{a_2b_1}$.
\end{remark}
\begin{definition}[Characteristic mapping]
Let $u$ and $v$ be two distinct points, and the line $(u,v)$ join
the points $p$ and $q$. We call $q ($or $p)$ the characteristic
mapping point of $p ($or $q)$ with respect to the basic points $u$
and $v$ if
$$[u,v;p,q]=1,$$
and denote $q=\chi_{(u,v)}(p)$ (or $p=\chi_{(u,v)}(q)$).
\end{definition}
Apparently if $q$ is the characteristic mapping point (or line) of
$p$, then $p$ is the characteristic mapping of $q$ as well. That is,
the characteristic mapping is reflexive, i.e., $\chi_{(u,v)}\circ
\chi_{(u,v)}=I$ (identity mapping). Geometrically, $\chi_{(u,v)}(p)$
and $p$ are symmetric with respect to the mid-point of $u$ and $v$.

From the definition of the characteristic mapping, Proposition 3.1
and Theorem 3.2, the property of the characteristic mapping can be
shown in the following corollaries.
\begin{corollary} Any three points $P, Q$ and $R$ in the
projective plane $\mathbb{P}^2$ are collinear if and only if their
characteristic mapping points $\chi_{(u,v)}(P),\chi_{(v,w)}(Q)$ and
$\chi_{(w,u)}(R)$ are collinear.
\end{corollary}

\begin{corollary} Any six distinct points $p_i \in
\mathbb{P}^2 (i=1,2,\cdots,6)$ lie on a conic if and only if the
image of their characteristic mapping $\chi_{(u,v)}(p_1)$,
$\chi_{(u,v)}(p_2)$, $\chi_{(v,w)}(p_3)$, $\chi_{(v,w)}(p_4)$,
$\chi_{(w,u)}(p_5)$ and $\chi_{(w,u)}(p_6)$ lie on a conic as well.
\end{corollary}

Bezout's theorem\cite{Walker50} says that {\em two algebraic curves
of degree $r$ and $s$ with no common components have exactly $r\cdot
s$ points in the projective complex plane}. In particular, a line
$l$ and an algebraic curve $C$ of degree $n$ without the component
$l$ meet in exactly $n$ points in the projective complex plane.

\begin{definition}[Characteristic number] Let $\Gamma_n$ be an algebraic curve of
degree $n$, and $a,b,c$ be any three distinct lines (without common
zero) where none of them is a component of $\Gamma_n$. Suppose that
there exist $n$ intersections between the each line and $\Gamma$,
and denoted by $\{p_i^{(a)},p_i^{(b)},$ $p_i^{(c)}\}_{i=1}^n$ the
intersections between $\Gamma_n$ and the lines $a,b,c$,
respectively. Let $u=<c,a>,v=<a,b>$, $w=<b,c>$.  The number
$$\mathcal {K}_n(\Gamma_n):=[u,v;p_1^{(a)},\cdots,p_n^{(a)}]\cdot
[v,w;p_1^{(b)},\cdots,p_n^{(b)}]\cdot[w,u;p_1^{(c)},\cdots,p_n^{(c)}],$$
independent of $a,b$ and $c$ (See Theorem 4.4 below), is called the
\textbf{\em characteristic number} of algebraic curve $\Gamma_n$ of
degree $n$.
\end{definition}

It is obvious from the Definition 3.9 that if $\Gamma_n$ is a
reducible curve of degree $n$ and has components $\Gamma_{n_1}$ and
$\Gamma_{n_2}, n=n_1+n_2$ , then $\mathcal K_n(\Gamma_n)=\mathcal
K_{n_1}(\Gamma_{n_1})\cdot \mathcal K_{n_1}(\Gamma_{n_1})$. From the
discussion below the Characteristic number is a global invariant of
algebraic curves.

By Definition 3.9, the characteristic numbers of line and conic are
-1 and +1 respectively.

\begin{definition}[Pascal mapping]
For any 6 points $p_1,p_2,\cdots,p_6$ without any three points are
collinear in the projective plane, first define $\Phi$ by
$$\Phi(\{p_1,p_2,\cdots,p_6\})=\{q_1,q_2,q_3\},$$
where $q_1=<(p_1,p_2),(p_4,p_5)>,q_2=<(p_2,p_3),(p_5,p_6)>$ and
$q_3=<(p_3,p_4),$ $(p_6,p_1)>$(i.e. $\{q_i\}_{i=1}^3$ are the three
pairs of the continuations of opposite side of the hexagon
determined by $\{p_i\}_{i=1}^6$). Then the {\em \textbf{Pascal
mapping}} $\Psi$ to $\{p_1,p_2,\cdots,p_6\}$ is defined by
$$\Psi\{p_1,p_2,\cdots,p_6\}:=\chi\circ \Phi\{p_1,p_2,\cdots,p_6\}
:=\{\chi_{(u,v)}(q_1),\chi_{(w,u)}(q_2),\chi_{(v,w)}(q_3)\},$$ where
$u=<(p_1,p_2),(p_5,p_6)>,v=<(p_1,p_2),(p_3,p_4)>$ and
$w=<(p_3,p_4),$ $(p_5,p_6)>$.
\end{definition}

Notice that the Pascal mapping on $p_1,p_2,\cdots,p_6$ giving above
depends on the order of $\{p_i\}_{i=1}^6$. One can also define the
Pascal mapping on $p_1,p_2,\cdots,p_6$ by $u=<(p_2,p_3),(p_4,p_5)>$,
$w=<(p_4,p_5),(p_1,p_6)>$ and $v=<(p_2,p_3),$ $(p_1,p_6)>$ instead,
which will not affect the result of the Pascal theorem (Theorem
3.11) giving below. But for the case of higher degrees as stated
below, we must insist on $u,v$ and $w$ being defined as in the
definition above.

\begin{figure}[h]
\begin{minipage}{0.45\textwidth}
\centering
  \includegraphics[width=0.9\textwidth]{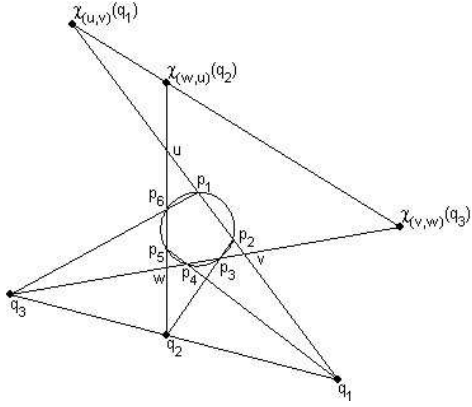}\\
\label{pascal} \caption{\small The Pascal mapping and Pascal
Theorem}
\end{minipage}
\begin{minipage}{0.5\textwidth}
Fig. 4 illustrates the Pascal mapping and the Pascal Theorem that
the three points
$\chi_{(u,v)}(q_1),\chi_{(w,u)}(q_2),\chi_{(v,w)}(q_3)$ are derived
by applying the Pascal mapping to $p_1,p_2,\cdots,p_6$. From
Corollary 3.7, we state the following version of Pascal theorem in
order to generalize it to cubic.
\begin{theorem}[Pascal theorem] For given 6 points
$p_1,p_2,\cdots,p_6$ on a conic section, the 3 points of image of
the Pascal mapping on these six points, $\Psi(\{p_i\}_{i=1}^6)$,
will all lie on a single line.
\end{theorem}
\end{minipage}
\end{figure}

\section{Invariant and Pascal Type Theorem}
In this section, we show our main results on the characteristic
number to algebraic curves. Consequently, a generalization of the
Pascal theorem to curves of higher degree is given by the
``principle of duality" and the spline method.

For cubic, we have
\begin{theorem}
The characteristic number of cubic is $-1$, that is, $\mathcal
{K}_3(\Gamma_3)=-1.$
\end{theorem}
\begin{proof}
See Appendix 5.2.
\end{proof}
Let $a,b$ and $c$ be any three distinct lines with no common zero in
the projective plane, denoted by $u=<c,a>,v=<a,b>,w=<b,c>$. Assume
that $p_1,p_2,p_3$ are three points on $a$, $p_4,p_5,p_6$ are on
$b$, and $p_7,p_8,p_9$ are on $c$, then there exist real numbers
$a_i,b_i, i=1,2,\cdots,9$ such that
\begin{eqnarray}
\left\{%
\begin{array}{ll}
    p_1=a_1u+b_1v \\
    p_2=a_2u+b_2v \\
    p_3=a_3u+b_3v \\
\end{array}%
\right. ,
  \left\{%
\begin{array}{ll}
    p_4=a_4v+b_4w \\
    p_5=a_5v+b_5w \\
    p_6=a_6v+b_6w \\
\end{array}%
\right.  and
 \left\{%
\begin{array}{ll}
    p_7=a_7w+b_7u \\
    p_8=a_8w+b_8u \\
    p_9=a_9w+b_9u,  \\
\end{array}%
\right.
\end{eqnarray}
Similar to Proposition 3.3, one can easily show, following the proof
of Theorem 4.1, that
\begin{proposition}
The nine points $p_1,p_2,\cdots,p_9$ lie on a cubic which differs
from $a\cdot b\cdot c=0$ if and only if
\begin{eqnarray}
\frac{ b_1b_2b_3}{a_1a_2a_3}\cdot \frac{b_4b_5b_6}{a_4a_5a_6}\cdot
\frac{b_7b_8b_9}{a_7a_8a_9}=-1,
\end{eqnarray}
holds.
\end{proposition}
Now, from Proposition 4.2, we provide in the following a new
generalization of the Pascal theorem to cubic.
\begin{theorem}
For any given 9 intersections between a cubic $\Gamma_3$ and any
three lines $a,b,c$ with no common zero, none of them is a component
of $\Gamma_3$, then the six points consisting of the three points
determined by the Pascal mapping applied to any six points (no three
points of which are collinear) among those 9 intersections as well
as the remaining three points of those 9 intersections must lie on a
conic.
\end{theorem}
\begin{proof}
Let $\{p_1,p_2,p_7\}=\Gamma_3\bigcap a$,
$\{p_3,p_4,p_8\}=\Gamma_3\bigcap b$,
$\{p_5,p_6,p_9\}=\Gamma_3\bigcap c$, and $u=<c,a>,v=<a,b>,w=<b,c>$.
Without loss of generality, we assume that $u=(1,0,0),v=(0,1,0)$ and
$w=(0,0,1)$. It is shown in Theorem 4.1 that those 9 points
$\{p_i\}_{i=1}^9 \in \mathbb{P}^2$ lying on a cubic implies
$\mathcal {K}_3(\Gamma_3)=-1$, or equivalently, (4.2) holds. Notice
that
\begin{center}
$
\begin{array}{ll}
q_1=<(p_1,p_2),(p_{4},p_{5})>=(b_{4}b_{5},-a_{4}a_{5},0)=b_{4}b_{5}u-a_{4}a_{5}v,\\
q_2=<(p_2,p_{3}),(p_{5}p_{6})>=(a_2a_{3},0,-b_2b_{3})=-b_2b_{3}w+a_2a_{3}u,\\
q_3=<(p_1,p_{6}),(p_{3},p_{4})>=(0,-b_1b_{6},a_1a_{6})=-b_1b_{6}v+a_1a_{6}w,\\
\end{array}
$
\end{center}
So applying the Pascal mapping on $p_1,p_2,p_3,p_4,p_5,p_6$, we have
$$\{\chi_{(u,v)}(q_1),\chi_{(w,u)}(q_2),\chi_{(v,w)}(q_3)\}=
\{-a_{4}a_{5}u+b_{4}b_{5}v, a_2a_{3}w-b_2b_{3}u, a_1a_{6}v-b_1b_{6}w
\}.$$ Since (4.2) is equivalent to
\begin{eqnarray}
(\frac{
b_4b_5}{-a_4a_5})\frac{b_7}{a_7}\cdot(\frac{-b_1b_6}{a_1a_6})\frac{b_8}{a_8}\cdot(\frac{-b_2b_3}{a_2a_3})\frac{b_9}{a_9}
=1.
\end{eqnarray}
Thus by Theorem 3.2 and Proposition 3.3, the six points
$\{\chi_{(u,v)}(q_1),$ $ \chi_{(w,u)}(q_2),$ $\chi_{(v,w)}(q_3),
p_7,p_8,p_9\}$ must lie on a conic.
\end{proof}
\begin{figure}[ht]
	\begin{center}
\includegraphics{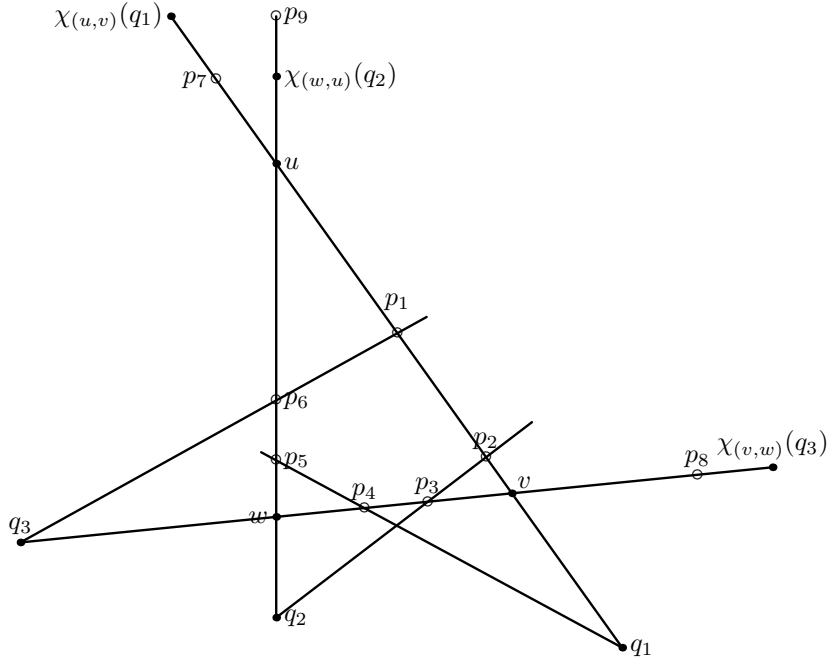}
\end{center}
\label{pascaltocubic} \caption{Generalization of Pascal Theorem}
\end{figure}
Theorem 4.3 implies that if $p_1,p_2,\cdots,p_9$ are intersection
points between a cubic and and any three distinct lines where non of
them is a component of the cubic (see Fig. 5), the three points
$\chi_{(u,v)}(q_1),\chi_{(w,u)}(q_2),\chi_{(v,w)}(q_3)$ along with
$p_7,p_8,p_9$ will lie on a conic. Obviously, this is an intrinsic
property of cubic!

Here, let us give an example to illustrate the Pascal type theorem
4.3. Let a cubic $\Gamma_3$ be given by
\begin{eqnarray*}
&&-1120x^3+560x^2y-60xy^2 +1008y^3-450xyz
+1200y^2z+580xz^2-1514yz^2\\
&&-729z^3=0,
\end{eqnarray*}
and three lines $a: x+z=0$, $b: -y+z=0$ and $c: -x+z=0$ be given.
Then the 9 intersections between $\Gamma_3$ and $a,b,c$ are
\begin{align*}
p_1&=(-4,-1,4), & p_2&=(-1,-\frac{3}{2},1),& p_3&=(\frac{1}{4},1,1),\\
p_4&=(-\frac{1}{4},1,1),& p_5&=(1,-\frac{3}{2},1),& p_6&=(1,-\frac{3}{4},1),\\
p_7&=(2,-1,-2),& p_8&=(\frac{1}{2},1,1),& p_9&=(1,\frac{47}{42},1),
\end{align*}
and $u=(0,-1,0), v=(-1,1,1), w=(1,1,1,)$. By direct computation, we
have
\begin{align*}
q_1&=<(p_1,p_2),(p_4,p_5)>=(-1,\frac{5}{2},1),\\
q_2&=<(p_2,p_3),(p_5,p_6)>=(1,\frac{5}{2},1),\\
q_3&=<(p_1,p_6),(p_3,p_4)>=(-6,1,1)
\end{align*}
and consequently
$$\chi_{(u,v)}(q_1)=(-1,\frac{5}{3},1), \chi_{(w,u)}(q_2)=(1,\frac{5}{3},1),\chi_{(v,w)}(q_3)=(6,1,1).$$
It is easy to verify that the six points $\chi_{(u,w)}(q_1),
\chi_{(v,u)}(q_2),\chi_{(w,v)}(q_3)$ as well as $p_7,p_8,p_9$ lie on
a conic:
$$4x^2+39xy-126y^2-65xz+312yz-174z^2=0.$$

In general, for algebraic curves of degree $n (n\geq 3)$, we have
proved the invariant of algebraic curves and the Pascal type theorem
to higher degrees. They are listed in the paper without proofs.
\begin{theorem} For any algebraic curve $\Gamma_n$ of degree $n$, its
characteristic number $\mathcal {K}_n(\Gamma_n)$ is always equal to
$(-1)^n.$
\end{theorem}

With this invariant, we may formulate a Pascal type Theorem for
algebraic curves of higher degrees:
\begin{theorem}[Pascal type Theorem] Let $a,b,c$ be any three distinct
lines with no common zero in the projective plane, and
$\{p_i^{(a)}\}_{i=1}^n$, $\{p_i^{(b)}\}_{i=1}^{n}$,
$\{p_i^{(c)}\}_{i=1}^{n}$ be given $n$ points lying on $a,b$ and
$c$, respectively. Then those $3n$ points $\{p_i^{(a)}\}_{i=1}^n$,
$\{p_i^{(b)}\}_{i=1}^{n}$, $\{p_i^{(c)}\}_{i=1}^{n}$ lie on an
algebraic curve of degree $n$ if and only if the $3(n-1)$ points
consisting of the three points determined by the Pascal mapping
applied to any six points (no three points of which are collinear)
among those $3n$ intersections as well as the remaining $3(n-2)$
points of those $3n$ intersections must lie on an algebraic curve of
degree $n-1$ as well.
\end{theorem}

In view of the simplicity of the invariant, some known results of
algebraic curves (see \cite{Walker50}, pp.123) can be easily
understood from our invariant (the characteristic number).
\begin{theorem}
If a line cuts a cubic in three distinct points, the residual
intersections of the tangents at these three points are collinear.
\end{theorem}
\begin{proof}
Let $l$ be a line cutting a given cubic $\Gamma_3$ at three points
$p_1,p_2,p_3$ and $l_1,l_2,l_3$ be the three tangents at these
points respectively. Denote by $q_1,q_2,q_3$ the residual
intersections between $\Gamma_3$ and $l_1,l_2,l_3$, respectively.
Let $u=<l_2,l_3>, v=<l_1,l_2>, w=<l_3,l_1>$. Then there are real
numbers $\{a_i,b_i,c_i,d_i\}_{i=1}^3$ such that
$p_1=a_1v+b_1w,p_2=a_2u+b_2v,p_3=a_3w+b_3u$ and $q_1=c_1v+d_1w,
q_2=c_2u+d_2v,q_3=c_3w+d_3u$. From Theorem 4.1 and Proposition 4.2,
we have
$$(\frac{b_1}{a_1})^2\frac{c_1}{d_1}\cdot(\frac{b_2}{a_2})^2\frac{c_2}{d_2}\cdot(\frac{b_3}{a_3})^2\frac{c_3}{d_3}=-1.$$
Since $p_1,p_2,p_3$ are collinear, then
$\frac{b_1}{a_1}\cdot\frac{b_2}{a_2}\cdot\frac{b_3}{a_3}=-1$. Hence,
we have
$\frac{c_1}{d_1}\cdot\frac{c_2}{d_2}\cdot\frac{c_3}{d_3}=-1$, and
those three points $q_1,q_2,q_3$ are collinear.
\end{proof}
\begin{theorem}
A line joining two flexes of a cubic passes through a third flexes.
\end{theorem}
\begin{proof}
Let $p_1,p_2,p_3$ be three flexes of a cubic, and $l_1,l_2,l_3$ be
the three tangents at these points. Let $u=<l_2,l_3>, v=<l_1,l_2>,
w=<l_3,l_1>$. Then there are real numbers $\{a_i,b_i\}_{i=1}^3$ such
that $p_1=a_1v+b_1w,p_2=a_2u+b_2v,p_3=a_3w+b_3u$. From Theorem 4.1,
we have
$(\frac{b_1}{a_1})^3\cdot(\frac{b_2}{a_2})^3\cdot(\frac{b_3}{a_3})^3=-1,$
hence $\frac{b_1}{a_1}\cdot\frac{b_2}{a_2}\cdot\frac{b_3}{a_3}=-1.$
Which implies $p_1,p_2,p_3$ are collinear.
\end{proof}
Similar to the proofs of Theorem 4.6 and Theorem 4.7, the following
theorem can be also easily proved by using the invariant that we
found.
\begin{theorem}
If a conic is tangent to a cubic at three distinct points, the
residual intersections of the tangents at these points are
collinear.
\end{theorem}

\section{Appendix}

\subsection{Bivariate Spline Space over Triangulations}
It is well known that spline is an important approximation tool in
computational geometry, and it is widely used in CAGD, scientific
computations and many fields of engineering. Splines, i.e.,
piecewise polynomials, forms linear spaces that have a very simple
structure in univariate case. However, it is quite complicated to
determine the structure of a space of bivariate spline over
arbitrary triangulation.

Bivariate spline is defined as follows\cite{W_Sh_L_S94}:
\begin{definition}
Let $\Omega$ be a given planar polygonal region and $\Delta$ be a
triangulation or partition of $\Omega$, denoted by $T_i,
i=1,2,\cdots,V$, called cells of $\Delta$. For integer $k>\mu\geq
0$, the linear space
$$S_k^\mu (\Delta):=\{s\mid s|_{T_i}\in \bf{\mathbb{P}_k}, s\in C^{\mu}(\Omega), \forall T_i \in \Delta \}$$
is called the spline space of degree $k$ with smoothness $\mu$,
where $\bf{\mathbb{P}_k}$ is the polynomial space of total degree
less than or equal to $k$.
\end{definition}

From the Smoothing Cofactor method\cite{W_Sh_L_S94}, the fundamental
theorem on bivariate splines was established.
\begin{theorem}
$s(x,y)\in S_k^\mu (\Delta)$ if and only if the following conditions
are satisfied:
\begin{enumerate}
\item[1.] For each interior edge of $\Delta$, which is defined by
$\Gamma_i:l_i(x,y)=0,$ there exists a so-called smoothing cofactor
$q_i(x,y),$ such that
\begin{equation*}
p_{i1}(x,y)-p_{i2}(x,y)=l_i^{\mu+1}(x,y)q_i(x,y),
\end{equation*}
where the polynomials $p_{i1}(x,y)$ and $p_{i2}(x,y)$ are determined
by the restriction of $s(x,y)$ on the two cells $\Delta_{i1}$ and
$\Delta_{i2}$ with $\Gamma_i$ as the common edge and $q_i(x,y)\in
\mathbb{P}_{k-(\mu+1)}$.
\item[2.] For any interior vertex $v_j$ of $\Delta$, the following
conformality conditions are satisfied
\begin{equation}
\sum[l_i^{(j)}(x,y)]^{\mu+1}q_i^{(j)}(x,y)\equiv 0,
\end{equation}
where the summation is taken on all interior edges $\Gamma_i^{(j)}$
passing through $v_j$, and the sign of the smoothing cofactors
$q_i^{(j)}$ are refixed in such a way that when a point crosses
$\Gamma_i^{(j)}$ from $\Delta_{i1}$ to $\Delta_{i2}$, it goes around
$v_j$ counter-clockwisely.
\end{enumerate}
\end{theorem}
From Theorem 5.2, the dimension of the space $S_k^\mu(\Delta)$ can
be expressed as
$$\dim S_k^\mu(\Delta)=\left(\begin{array}{c} k+2\\ 2 \end{array}\right)+\tau,$$
where $\tau$ is the dimension of the linear space defined by the
conformality conditions (5.1).

However, for an arbitrary given triangulation, the dimension of
these spaces depends not only on the topology of the triangulation,
but also on the geometry of the triangulation. In general cases, no
dimension formula is known. We say that a triangulation is {\em
singular} to $S_k^\mu(\Delta)$ if the dimension of the spline space
depends on, in additional to the topology of the triangulation, the
geometric position of the vertices of $\Delta$, and
$S_k^\mu(\Delta)$ is singular when its dimension increases according
to the geometric property of $\Delta$. Hence, the singularity of
multivariate spline spaces is an important object that is inevitable
in the research of the structure of multivariate spline spaces. For
example, Morgan and Scott's triangulation
$\Delta_{MS}$(\cite{M-S75}, see Fig. 6) is singular to
$S_2^1(\Delta_{MS})$. That is to say that the dimension of the space
$S_2^1(\Delta_{MS})$ is 6 in general but it increases to 7 when the
position of the inner vertices satisfy certain conditions.

\begin{figure}[h]
\begin{minipage}{0.48\textwidth}
\includegraphics[width=1.0\textwidth]{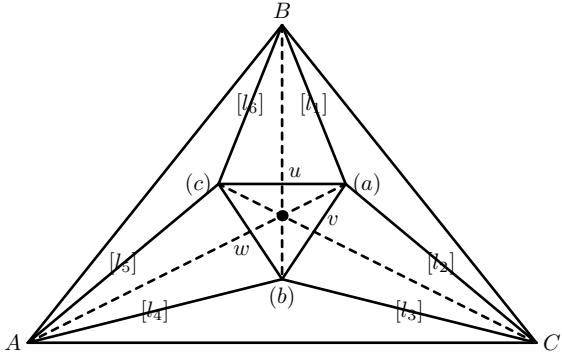}\\
\label{MS}\caption{Morgan-Scott triangulation}
\end{minipage}
\begin{minipage}{0.5\textwidth}
While the singularity of multivariate spline over any triangulation
has not been completely settled, many results on the structure of
multivariate spline space in the past 30 years can be found in many
of references {\cite{Alf85,Alf87,Alf_Sch87,Alf_P_Sch87,Alf_P_
Sch871,Bil88,C_Wang83,C_Wang831,C_S_W85,Dah_Mic89,Dah_Mic96,Dien90,Sch84,
Sch79,R_H_Lu88,W_Sh_L_S94,Luo_Wang05}}. For Morgan-Scott's
triangulation, Shi\cite{shixq91} and Diener\cite{Dien90}
independently obtained the geometric significance of the necessary
and sufficient condition of $\dim S_2^1(\Delta_{MS})=7$,
respectively, and an equivalent geometric necessary and sufficient
condition of singularity of $S_2^1(\Delta_{MS})$ from the viewpoint
of projective geometry was obtained in \cite{Du03}.
\end{minipage}
\end{figure}

Now, we take an example for $S_2^1(\Delta_{MS})$ to intuitively
understand Theorem 5.2. Let $l_i: \alpha_i x+\beta_i y+\gamma_i z=0
\ \ (i=1,2,\cdots,6)$, $u:\alpha_u x+\beta_u y+\gamma_u
z=0$,$v:\alpha_v x+\beta_v y+\gamma_v z=0$ and $w:\alpha_w x+\beta_w
y+\gamma_w z=0$ in Morgan-Scott triangulation shown in Fig. 6. From
Theorem 5.2, the global conformality condition in
$S_2^1(\Delta_{MS})$ is
\begin{eqnarray}
\left\{\begin{array}{ll}
\lambda_1l_1^2+\lambda_2l_2^2+\lambda_uu^2+\lambda_vv^2=0,\\
\lambda_3l_3^2+\lambda_4l_4^2-\lambda_vv^2+\lambda_ww^2=0,\\
\lambda_5l_5^2+\lambda_6l_6^2-\lambda_ww^2-\lambda_uu^2=0,
\end{array}
\right.
\end{eqnarray}
where all letters of $\lambda's$ are undetermined real constants.
Then the $\dim S_2^1(\Delta_{MS})=6+\tau$, where $\tau$ is the
dimension of the linear space defined by (5.2). However, the
structure of $S_2^1(\Delta_{MS})$ depends on the geometric positions
of the inner vertices $a,b$ and $c$, which can be obviously shown
from the following conclusions.
\begin{theorem}[{\cite{shixq91}}] The
spline space $S_2^1(\Delta_{MS})$ is singular (i.e. $\dim
S_2^1(\Delta_{MS})=7$) if and only if $Aa,Bb,Cc$ are concurrent,
otherwise $\dim S_2^1(\Delta_{MS})=6$(see Fig.6).
\end{theorem}
\begin{theorem}[{\cite{Du03}}] Let $l_i(x,y,z)=\alpha_i x+\beta_i y+\gamma_i
z=0$ $(i=1,2,\cdots,6)$, then the spline space $S_2^1(\Delta_{MS})$
is singular (i.e. $\dim S_2^1(\Delta_{MS})=7$) if and only if 6
points $\{(\alpha_i, \beta_i, \gamma_i)\}_{i=1}^6$ lie on a conic,
otherwise $\dim S_2^1(\Delta_{MS})=6$.
\end{theorem}
Using the principle of duality, an interesting fact is that the
equivalent relations in Theorem 5.3 and Theorem 5.4 hold because of
the Pascal theorem!

More Precisely, for the Morgan-Scott triangulation, let
\begin{eqnarray}
\left\{%
\begin{array}{ll}
    l_1=a_1u+b_1v \\
    l_2=a_2u+b_2v \\
\end{array}%
\right. ,
  \left\{%
\begin{array}{ll}
    l_3=a_3v+b_3w \\
    l_4=a_4v+b_4w \\
\end{array}%
\right.  and
 \left\{%
\begin{array}{ll}
    l_5=a_5w+b_5u \\
    l_6=a_6w+b_6u, \\
\end{array}%
\right.
\end{eqnarray}
where all $a_i's$ and $b_i's$ are constants, then by solving the
system of equations in (5.2), we have
\begin{theorem}[{\cite{Luo01}},{\cite{shixq91}}]
The spline space $S_2^1(\Delta_{MS})$ is singular (i.e. $\dim
S_2^1(\Delta_{MS})=7$) if and only if
\begin{eqnarray}
\frac{ b_1b_2}{a_1a_2}\cdot \frac{b_3b_4}{a_3a_4}\cdot
\frac{b_5b_6}{a_5a_6}=1.
\end{eqnarray}
\end{theorem}

\begin{figure}[h]
\begin{minipage}{0.45\textwidth}
\centering
  \includegraphics[width=0.8\textwidth]{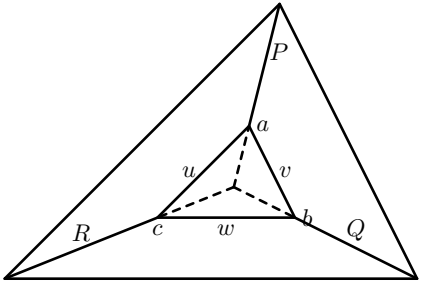}\\
\caption{Partition $\Delta$}
\end{minipage}
\begin{minipage}{0.5\textwidth}
\begin{remark}
In fact, there also exists the singularity in the simplest spline
space $S_1^0(\Delta)$ consisting of continuous piecewise linear
polynomials over arbitrary partition $\Delta$. For instance, let
$\Delta$ be a partition shown in Fig. 7, the dual figure of $\Delta$
is in Fig. 3. Using the same notations in Proposition 3.1, it is
easy to verify through the duality principle that $\dim
S_1^0(\Delta)=4$ when $\frac{b_1}{a_1}\cdot \frac{b_2}{a_2}\cdot
\frac{b_3}{a_3}=-1$, otherwise $\dim S_1^0(\Delta)=3.$
\end{remark}
\end{minipage}
\end{figure}
In general, for $\mu \geq 3$, Luo \& Chen\cite{Luo_Chen05} gave an
equivalent condition in an algebraic form to the singularity of
$S_{\mu+1}^{\mu}(\Delta_{MS}^{\mu}) (\mu \geq 3)$ as follows: for a
given triangulation $\Delta_{MS}^{\mu}$(see Fig. 9), suppose
\begin{eqnarray}
\left\{
\begin{array}{ll}
   l_{i}=a_{i}u+b_{i}v,\qquad i=1,2,\ldots\ldots\mu+1,\\
   l_{j}=a_{j}v+b_{j}w,\qquad j=\mu+2,\mu+3,\ldots\ldots2\mu+2,\\
   l_{k}=a_{k}w+b_{k}u,\qquad k=2\mu+3,2\mu+4,\ldots\ldots3\mu+3, \\
\end{array}
\right.
\end{eqnarray}
then
\begin{theorem}[{\cite{Luo_Chen05}}]
The spline space $S_{\mu+1}^{\mu}(\Delta_{MS}^{\mu})$ is singular if
and only if
\begin{equation}
\frac{a_1\ldots\ldots a_{\mu+1}}{b_1\ldots\ldots
b_{\mu+1}}\cdot\frac{a_{\mu+2}\ldots\ldots
a_{2\mu+2}}{b_{\mu+2}\ldots\ldots
b_{2\mu+2}}\cdot\frac{a_{2\mu+3}\ldots\ldots
a_{3\mu+3}}{b_{2\mu+3}\ldots\ldots b_{3\mu+3}}={(-1)}^{\mu+1}.
\end{equation}
\end{theorem}

\begin{figure}[h]
\begin{minipage}[b]{0.45\textwidth}
\centering
  \includegraphics[width=0.8\textwidth]{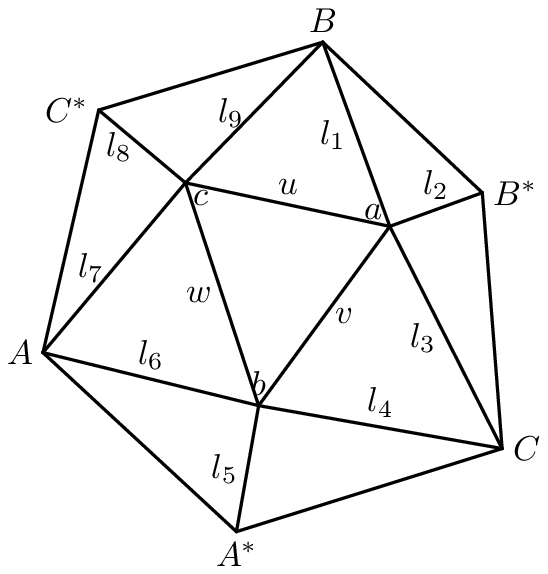}\\
\caption{Morgan-Scott's type triangulation $\Delta_{MS}^2$}
\end{minipage}
\begin{minipage}[b]{0.45\textwidth}
\centering
  \includegraphics[width=0.8\textwidth]{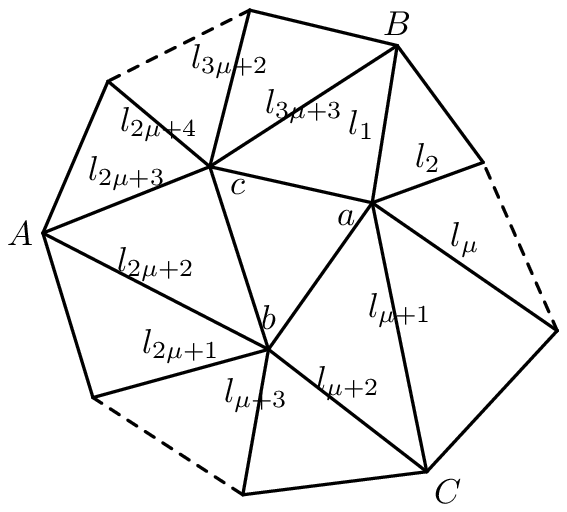}\\
\caption{Morgan-Scott's type triangulation $\Delta_{MS}^\mu$}
\end{minipage}
\end{figure}

For the geometric condition of the singularity of
$S_3^2(\Delta_{MS}^2)$, it was analyzed in {\cite{Luo_Wang05}} from
projective geometry point of view and the following result was
obtained.

Let $l_i: \alpha_ix+\beta_iy+\gamma_i z=0 (i=1,2,\ldots,9)$,
$a=(a_1,a_3,a_3), b=(b_1,b_2,b_3)$, and $c=(c_1,c_2,c_3)$ in
$\Delta_{MS}^2$ triangulation (see Fig.5.3). Let
$l_a=a_1x+a_2y+a_3z, l_b=b_1x+b_2y+b_3z$ and $l_c=c_1x+c_2y+c_3z$.
We define $\bar{\mathbb{P}}_3$ to be the cubic polynomial subspaces
spanned by any nine monomials of $\{x^3,y^3,z^3,x^2y,xy^2,$
$y^2z,yz^2,x^2z,xz^2,xyz\}$ as in \cite{Luo_Wang05}.
\begin{theorem}[{\cite{Luo_Wang05}}]
The spline space $S_3^2(\Delta_{MS}^2)$ is singular (i.e. $dim S_3^2
(\Delta_{MS}^2)=11$) if and only if
$p_i=(\alpha_i,\beta_i,\gamma_i),(i=1,2,\ldots,9)$ lie on a plane
curve, which differs from $l_a\cdot l_b\cdot l_c=0$, in
$\bar{\mathbb{P}}_3$.
\end{theorem}
\begin{proof}
Embedding $\Delta_{MS}^2$ to $\mathbb{P}^2$ by the map:
$(x,y)\longmapsto [x,y,1]$. Suppose the lines $\bar{bc}, \bar{ca}$
and $\bar{ab}$ are given by $u=0,v=0$ and $w=0$, respectively. There
are real numbers $a_i, b_i (i=1,2,\cdots,9)$ such that
\begin{eqnarray}
\left\{%
\begin{array}{ll}
    l_1=a_1u+b_1v \\
    l_2=a_2u+b_2v \\
    l_3=a_3u+b_3v
\end{array}%
\right. ,
  \left\{%
\begin{array}{ll}
    l_4=a_4v+b_4w \\
    l_5=a_5v+b_5w \\
    l_6=a_6v+b_6w
\end{array}%
\right.  and
 \left\{%
\begin{array}{ll}
    l_7=a_7w+b_7u \\
    l_8=a_8w+b_8u \\
    l_9=a_9w+b_9u.
\end{array}%
\right.
\end{eqnarray}
Let $\lambda_i(i=1,2,\ldots,9)$ be the corresponding smoothing
cofactors and let $p_i=(\alpha_i,\beta_i,\gamma_i),i=1,2,\ldots,9$.
Then the global conformality conditions in $S_3^2(\Delta_{MS}^2)$
become
\begin{eqnarray}
\left\{\begin{array}{ll}
\lambda_1l_1^3+\lambda_2l_2^3+\lambda_3l_3^3+\lambda_uu^3+\lambda_vv^3=0,\\
\lambda_4l_4^3+\lambda_5l_5^3+\lambda_6l_6^3-\lambda_vv^3+\lambda_ww^3=0,\\
\lambda_7l_7^3+\lambda_8l_8^3+\lambda_9l_9^3-\lambda_ww^3-\lambda_uu^3=0.
\end{array}
\right.
\end{eqnarray}
Let
$\psi:(\lambda_1,\lambda_2,\cdots,\lambda_9,\lambda_u,\lambda_v,\lambda_w)\longmapsto
(\lambda_1,\lambda_2,\cdots,\lambda_9)$, then $\psi$ is an injective
linear map from the solution spaces of (5.7) to the solution space
of
\begin{eqnarray}
\sum_{i=1}^9 \lambda_il_i^3(x,y,z)=0.
\end{eqnarray}
Similar to \cite{Du03}, we intend to prove that the map $\psi$ is
also bijective. For this purpose, let
$\Lambda=(\lambda_1,\lambda_2,\cdots,\lambda_9)$ be a solution of
(5.9). Taking $u=0$, $w=0$ and $v=0$ in (5.9) respectively, we have
\begin{eqnarray*}
\begin{array}{lll}
\lambda_4l_4^3+\lambda_5l_5^3+\lambda_6l_6^3+k_1w^3+k_2v^3=0,\\
\lambda_1l_1^3+\lambda_2l_2^3+\lambda_3l_3^3+k_3u^3+k_4v^3=0,\\
\end{array}
\end{eqnarray*}
and
\begin{eqnarray*}
\begin{array}{l}
\lambda_7l_7^3+\lambda_8l_8^3+\lambda_9l_9^3+k_5u^3+k_6w^3=0,
\end{array}
\end{eqnarray*}
where all $k_i (i=1,2,\cdots,6)$ are real numbers determined by
$\Lambda$ and the coefficients in (5.9). Since $\Lambda$ is a
solution of (5.9), it follows that
$$k_1w^3+k_2v^3+k_3u^3+k_4v^3+k_5u^3+k_6w^3=0,$$
and $k_1=-k_6, k_2=-k_4, k_3=-k_5$. Consequently,
$\tilde{\Lambda}:=(\lambda_1,\cdots,\lambda_9, k_3, k_2, k_1)$ is a
solution of (5.8).

Hence, $\dim S_3^2(\Delta_{MS}^2)=11$(or $S_3^2(\Delta_{MS}^2)$ is
singular) if and only if there exists a nonzero solution of equation
(5.9). Now expand (5.9) with respect to $x,y,z$, will result in a
system of linear equations:
\begin{eqnarray}
\hspace{0.8cm}\mathbb{M}\Lambda :=\left(
  \begin{array}{ccccc}
    \alpha_1^3 & \alpha_2^3 & \cdots & \alpha_8^3 & \alpha_9^3 \\
    \alpha_1^2\beta_1 & \alpha_2^2\beta_2 & \cdots & \alpha_8^2\beta_8 & \alpha_9^2\beta_9 \\
    \cdots & \cdots & \cdots & \cdots & \cdots \\
    \alpha_1\beta_1\gamma_1 & \alpha_2\beta_2\gamma_2 & \cdots & \alpha_8\beta_8\gamma_8 & \alpha_9\beta_9\gamma_9 \\
    \cdots & \cdots & \cdots & \cdots & \cdots \\
    \beta_1\gamma_1^2 & \beta_2\gamma_2^2 & \cdots & \beta_8\gamma_8^2 & \beta_9\gamma_9^2 \\
    \gamma_1^3 & \gamma_2^3 & \cdots & \gamma_8^3 & \gamma_9^3 \\
  \end{array}
\right)\cdot
\left(
  \begin{array}{c}
    \lambda_1 \\
    \lambda_2 \\
    \vdots \\
    \lambda_9 \\
  \end{array}
\right)=0
\end{eqnarray}
Notice that $p_i=(\alpha_i,\beta_i,\gamma_i)(i=1,2,\ldots,9)$ lie on
the cubic $C_3:=l_a\cdot l_b\cdot l_c=0$, obviously the row vectors
of the coefficient matrix are linearly dependent. Since no four
points in $\{p_i=(\alpha_i,\beta_i,\gamma_i)\}_{i=1}^9$ are
collinear, it can be shown from a classical results of algebraic
geometry that rank$(\mathbb{M})\geq 8$. Hence, (5.10) has a non-zero
solution $\Lambda$ if and only if the rank of the coefficient matrix
of (5.10) is equal to 8, implying that those nine points
$p_i=(\alpha_i,\beta_i,\gamma_i)(i=1,2,\ldots,9)$ lie on a cubic in
$\bar{\mathbb{P}}_3$. Conversely, let
$\{p_i=(\alpha_i,\beta_i,\gamma_i)\}_{i=1}^9$ lie on a cubic
$\Gamma_3$ in $\bar{\mathbb{P}}_3$, and $\Gamma_3$ differ from
$C_3$. Without loss of generality, suppose
$$\Gamma_3:\ \ a_1x^3+a_2x^2y+a_3xy^2+a_4y^3+a_5xyz+a_6x^2z+a_7xz^2+a_8y^2z+a_9yz^2=0$$
(no $z^3$ term),then we claim that $C_3$ must contain a $z^3$ term.
Otherwise, by simple computation, there exists constant $d$ such
that a cubic $\bar{\Gamma}_3=\Gamma_3+dC_3$, composed some 8 basis
elements in $\{x^3,y^3,z^3,x^2y,xy^2,y^2z,yz^2,x^2z,xz^2,xyz\}$,
passes through the nine points
$\{p_i=(\alpha_i,\beta_i,\gamma_i)\}_{i=1}^9$. Thus the rank of the
coefficient matrix $\mathbb{M}$ must be less than 8, which is
contradictory. Since $p_i (i=1,2,\cdots,9)$ lie on $\Gamma_3$,
\begin{eqnarray}
\left(
  \begin{array}{ccccc}
    \alpha_1^3 & \alpha_1^2\beta_1 & \cdots & \beta_1^2\gamma_1 & \beta_1\gamma_1^2 \\
    \alpha_2^3 & \alpha_2^2\beta_2 & \cdots & \beta_2^2\gamma_2 & \beta_2\gamma_2^2 \\
    \cdots &  & \cdots &  & \cdots \\
    \alpha_9^3 & \alpha_9^2\beta_9 & \cdots & \beta_9^2\gamma_9 & \beta_9\gamma_9^2 \\
  \end{array}
\right)\cdot\left(
              \begin{array}{c}
                a_1 \\
                a_2 \\
                \vdots \\
                a_9 \\
              \end{array}
            \right)=0.
\end{eqnarray}
Obviously, the system of linear equations:
\begin{eqnarray}
\left(
  \begin{array}{ccccc}
    \alpha_1^3 & \alpha_2^3 & \cdots & \alpha_8^3 & \alpha_9^3 \\
    \alpha_1^2\beta_1 & \alpha_2^2\beta_2 & \cdots & \alpha_8^2\beta_8 & \alpha_9^2\beta_9 \\
    \cdots & \cdots & \cdots & \cdots & \cdots \\
    \alpha_1\beta_1\gamma_1 & \alpha_2\beta_2\gamma_2 & \cdots & \alpha_8\beta_8\gamma_8 & \alpha_9\beta_9\gamma_9 \\
    \cdots & \cdots & \cdots & \cdots & \cdots \\
    \beta_1\gamma_1^2 & \beta_2\gamma_2^2 & \cdots & \beta_8\gamma_8^2 & \beta_9\gamma_9^2 \\
  \end{array}
\right)\cdot \left(
  \begin{array}{c}
    \lambda_1 \\
    \lambda_2 \\
    \vdots \\
    \lambda_9 \\
  \end{array}
\right)=0
\end{eqnarray}
has a non-zero solution. The condition of $C_3$ containing a term
$z^3$ and passing through $p_i (i=1,2,\cdots, 9)$ show that the
vector $(\gamma_1^3,\gamma_2^3,\cdots,\gamma_9^3)$ can be expressed
by the linear combination of the 9 row vectors in (5.12). Therefore,
the non-zero solution $\Lambda$ of (5.12) is also solution of (5.9)
and (5.8). This completes the proof.
\end{proof}
\begin{remark}
In fact, it can be easily seen from the process of the proof of
Theorem 5.8 or from the Chasles's Theorem that Theorem 5.8 can be
improved as:{\em The spline space $S_3^2(\Delta_{MS}^2)$ is singular
(i.e. $dim S_3^2 (\Delta_{MS}^2)=11$) if and only if
$p_i=(\alpha_i,\beta_i,\gamma_i),(i=1,2,\ldots,9)$ lie on a cubic,
which differs from $l_a\cdot l_b\cdot l_c=0$.}
\end{remark}

\subsection{Proof of Theorem 4.1}
Let $a,b$ and $c$ be any three distinct lines in the projective
plane $\mathbb{P}^2$, denoted by $u=<c,a>,v=<a,b>$ and $w=<b,c>$,
and $\Gamma_3$ be a cubic in $\mathbb{P}^2$. Assume that
$p_1,p_2,p_3$ are three intersection points of $a$ and $\Gamma_3$,
$p_4,p_5,p_6$ are intersection points of $b$ and $\Gamma_3$, and
$p_7,p_8,p_9$ are intersection points of $c$ and $\Gamma_3$. Then
there are real numbers $\{a_i,b_i\}$ such that
\begin{eqnarray*}
\left\{%
\begin{array}{ll}
    p_1=a_1u+b_1v \\
    p_2=a_2u+b_2v \\
    p_3=a_3u+b_3v
\end{array}%
\right. ,
  \left\{%
\begin{array}{ll}
    p_4=a_4v+b_4w \\
    p_5=a_5v+b_5w \\
    p_6=a_6v+b_6w
\end{array}%
\right.  and
 \left\{%
\begin{array}{ll}
    p_7=a_7w+b_7u \\
    p_8=a_8w+b_8u \\
    p_9=a_9w+b_9u.
\end{array}%
\right.
\end{eqnarray*}
Using Definition 1.3, the duality of the figure composed by the
lines $a,b$ and $c$, the points $u,v$, $w$ and $\{p_i\}_{i=1}^9$
turns precisely out the Morgan-Scott type partition $\Delta_{MS}^2$
(in which $\mu=2$) as shown in Fig.5.3, where
$l_i=\alpha_ix+\beta_iy+\gamma_iz, i=1,2,\cdots,9$. From Theorem
5.8, we see that the spline space $S_3^2(\Delta_{MS}^2)$ is
singular, that is, $\dim S_3^2(\Delta_{MS}^2)=11$. Consequently, it
follows from Theorem 5.7 ($\mu=2$) that the characteristic number of
a cubic is always equal to $(-1)^3=-1$. Which completes the proof of
our main result.

\ \ \\
\textbf{Acknowledgement.} The author would like to appreciate Prof.
Prof. T.Y. Li of Michigan State University for his valuable comments
and suggestions, that helped to improve the paper. The author also
thanks Prof. R.H. Wang and J.Z. Nan of Dalian University of
Technology for their kind suggestions from the aspect of
computational geometry and projective geometry.

\end{document}